\documentclass{amsart}

\usepackage{amsmath,amssymb,amsthm,amsfonts,mathrsfs,mathtools}
\usepackage[frame,cmtip,arrow,matrix,line,graph,curve]{xy}
\usepackage[mathscr]{eucal}
\usepackage{mathabx}
\usepackage[colorlinks,backref=page,citecolor=blue]{hyperref}
\usepackage{tikz}
\usepackage{epic,eepic}
\usepackage{yfonts}
\usepackage{enumerate}
\usepackage{comment}

\theoremstyle{definition}
\newtheorem{mainthm}{Theorem}

\newtheorem{theorem}{Theorem}[section]
\newtheorem{definition}[theorem]{Definition}
\newtheorem{conjecture}[theorem]{Conjecture}
\newtheorem{lemma}[theorem]{Lemma}
\newtheorem{proposition}[theorem]{Proposition}
\newtheorem{corollary}[theorem]{Corollary}

\newtheorem*{theorem*}{Theorem}

\theoremstyle{remark}

\newtheorem{example}[theorem]{Example}

\def\RR{\mathbb{R}}

\def\T{\mathbb{T}}
\def\L{\mathbb{L}}
\def\BR{\mathrm{BR}}
\def\Sym{\mathrm{Sym}}
\def\Cut{\mathrm{Cut}}

\DeclareMathOperator{\Vol}{Vol}

\begin{document}

\title{Bounded ratios for Lorentzian matrices}


\author{Daoji Huang}
\address{\rm Daoji Huang, University of Massachusetts Amherst}
\email{daojihuang@umass.edu}

\author{June Huh}
\address{\rm June Huh, Princeton University and Korea Institute for Advanced Study}
\email{huh@princeton.edu}

\author{Daniel Soskin}
\address{\rm Daniel Soskin, University of California, Los Angeles}
\email{dsoskin@math.ucla.edu}

\author{Botong Wang}
\address{\rm Botong Wang, University of Wisconsin–Madison}
\email{wang@math.wisc.edu}

\begin{abstract}
We study multiplicative inequalities among entries of Lorentzian matrices, referred to as \emph{bounded ratios}.  
These inequalities can be viewed as generalizations of  the classical Alexandrov--Fenchel inequalities for mixed volumes.
Our main structural result identifies the cone of all bounded ratios on Lorentzian matrices with the dual of the \emph{cut cone}, a finitely generated integral polyhedral cone extensively studied in metric geometry and graph theory. 
We examine in detail the \emph{pentagonal ratio}, which first appears for Lorentzian matrices of size at least five.
For Lorentzian matrices of size three, we determine the optimal bounding constants across the entire cone of bounded ratios, obtaining an explicit entropy-like formula. 
\end{abstract}
\maketitle

\section{introduction}

A \emph{Lorentzian matrix} is a symmetric matrix with nonnegative real entries that has at most one positive eigenvalue.
We denote by $\L_n$ the set of $n\times n$ Lorentzian matrices.
These matrices appear prominently in scattering problems in quantum field theory and in the theory of mixed volumes in convex geometry, see \cite[Section 6]{Kinematic} and \cite[Chapter 5]{Schneider}. In this work, we study \emph{multiplicative inequalities} for Lorentzian matrices---inequalities of the form
\[
\prod_{i \le j} p_{ij}^{\alpha_{ij}} \le c \prod_{i \le j} p_{ij}^{\beta_{ij}} \ \ \text{for all $(p_{ij}) \in \L_n$,}
\]
where $\alpha_{ij}$ and $\beta_{ij}$ are nonnegative integers, and $c$ is a positive constant.

For a collection of convex bodies $K=(K_1, \dots, K_d)$ in $\RR^d$, the \emph{mixed volume} of $K$ is defined to be the normalized coefficient of $t_1\cdots t_d$ in the associated volume polynomial:
\[
V(K_1, \dots, K_d)\coloneq \frac{1}{d!}\partial_1\cdots \partial_d \Vol_d(t_1K_1+\dots +t_dK_d).
\]
The celebrated \emph{Alexandrov--Fenchel inequality} states that,
for any collection of convex bodies 
\[
 P_1,\ldots,P_n \ \ \text{and} \ \ K_1,\ldots,K_{d-2} \ \ \text{in $\RR^d$,}
\]
the symmetric $n \times n$ matrix of $d$-dimensional mixed volumes $(p_{ij})$ given by
\[
p_{ij}=V(P_i, P_j, K_1, \dots, K_{d-2}) 
\]
is a Lorentzian matrix \cite[Section 7.3]{Schneider}. In particular, any of its $2 \times 2$ principal minors is nonpositive, that is, 
\[
p_{ii}p_{jj} \le p_{ij}^2 \ \ \text{for all $i$ and $j$.}
\]

The Alexandrov--Fenchel inequality above is one of the first examples of bounded ratios on Lorentzian matrices, which we now define. Let $\Sym_n(\RR_{> 0})$ denote the set of $n \times n$ symmetric matrices with positive entries, and let $X \subseteq \Sym_n(\RR_{> 0})$ be a nonempty subset.
Given a real vector $\alpha=(\alpha_{ij})_{1 \le  i \le j \le n}$, we write
\[
\alpha=\sum_{i \le j} \alpha_{ij} \mathbf{e}_{ij},
\]
where $\{\mathbf{e}_{ij}\}_{i \le j}$ is the standard basis of the space of all such vectors. 

\begin{definition}
The set of \emph{bounded ratios} on $X$, denoted $\BR(X)$, is the set of all real vectors $\alpha=(\alpha_{ij})_{1 \le  i \le j \le n}$
that satisfy the following condition:
\[
\text{There is a positive constant $c$ such that $\prod_{1\le i\le j\le n}p_{ij}^{\alpha_{ij}}\le c$ for all matrices $(p_{ij})$ in $X$.}
\]
We define the \emph{optimal bounding constant} of a bounded ratio $\alpha$ on $X$, denoted $f(\alpha)=f_X(\alpha)$,  to be the infimum of all possible constants $c$ satisfying the displayed inequality for all $(p_{ij})$ in $X$. 
\end{definition}

It is straightforward to check that $\BR(X)$ is a convex cone in $\RR^{\binom{n+1}{2}}$ and $\log f(\alpha)$ is a convex homogeneous function of degree $1$ on the cone of bounded ratios. 
When $X$ is the set of $n\times n$ positive semidefinite matrices with positive entries $\operatorname{PSD}^+_n$, 
Yu computes the tropicalization of $X$ in \cite{Yu}, which implies that the cone of bounded ratios on $X$  is a finitely generated integral polyhedral cone with $\binom{n}{2}$ extremal rays corresponding to the $2 \times 2$ principal minors:\footnote{After the completion of this paper, Lorenzo Baldi and Mario Kummer in \cite{BaldiKummer} proved the general result that the set of bounded ratios $\BR(X)$ on a semialgebraic set $X$ in $\mathbb{R}^n_{>0}$ is the convex cone of linear forms that are nonnegative on the tropicalization of $X$; see also \cite{ARVY}. When $X$ is the set of Lorentzian polynomials with fixed $\mathrm{M}$-convex support, $\BR(X)$ is the dual to the set of $\mathrm{M}$-convex functions \cite{BaldiKummer,BRS}. Theorem~\ref{mth:BR} is  the special case when $X$ is the set of quadratic Lorentzian polynomials.}
\[
\BR(\operatorname{PSD}^+_n)=\operatorname{Cone}(-\mathbf{e}_{ii}-\mathbf{e}_{jj}+2\mathbf{e}_{ij})_{1 \le i<j \le n}.
\]
In contrast, bounded ratios for Lorentzian matrices, as we will see, exhibit more intricate behaviors. 
For related notions of bounded ratios in the study of cluster algebras and total positivity, see \cite{Fallat-Johnson,Fallat-Gekhtman-Johnson,Gekhtman-Greenberg-Soskin,Gekhtman-Soskin}. In \cite{Fallat-Johnson}, Fallat and Johnson introduced bounded ratios as a method for studying multiplicative determinantal inequalities. This approach proved effective for classifying multiplicative inequalities among principal minors of totally positive matrices \cite{Fallat-Gekhtman-Johnson}, as well as for studying multiplicative inequalities among Pl\"ucker coordinates on the positive Grassmannian \cite{BoFr,Gekhtman-Soskin} and among cluster variables on the positive locus of full rank cluster algebras of finite type \cite{Gekhtman-Greenberg-Soskin}.  

\begin{definition}
A bounded ratio $\alpha$ is \emph{integral} if all the entries $\alpha_{ij}$ are integers. 
An integral bounded ratio on $X$ is said to be \emph{primitive} if it is the primitive integral generator of an extremal ray of $\BR(X)$.
\end{definition}

Let $\L_n^+$ be the set of $n\times n$ Lorentzian matrices with only positive entries. Since any Lorentzian matrix is a limit of Lorentzian matrices with positive entries \cite[Section 2]{Branden-Huh}, the multiplicative inequalities for $\L_n$ are obtained by clearing the denominator in $p^\alpha \le c$ for a bounded ratio $\alpha$ on $\L_n^+$.
Our main result, Theorem~\ref{mth:BR} below, implies that $\BR(\L_n^+)$ is a finitely generated integral polyhedral cone. In particular, as a polyhedral cone, $\BR(\L_n^+)$ is generated by primitive integral bounded ratios. 

\begin{example}
The bounded ratios on $\L_2^+$ are exactly the nonnegative multiples of 
\[
\alpha^{12}\coloneq \mathbf{e}_{11}+\mathbf{e}_{22}-2\mathbf{e}_{12}.
\]
In other words, $\BR(\L_2^+)=\operatorname{Cone}(\alpha^{12})$. 
The optimal bounding constants are determined by the condition $f(\alpha^{12})=1$. We refer to the bounded ratio $\alpha^{12}$ as the \emph{Alexandrov--Fenchel type}.
\end{example}

\begin{example}
Theorem~\ref{mth:BR} shows that 
there are precisely three primitive bounded ratios on $\L_3^+$, corresponding to the inequalities
\[
\frac{p_{23}p_{11}}{p_{12}p_{13}} \le 2, \quad
\frac{p_{13}p_{22}}{p_{12}p_{23}} \le 2, \quad
\frac{p_{12}p_{33}}{p_{13}p_{23}} \le 2 \quad \text{for all $(p_{ij}) \in \L_3^+$.}
\]
The cone of bounded ratios $\BR(\L_3^+)$ is the three-dimensional simplicial cone in $\RR^6$ generated by
\[
\alpha^{23|1}\coloneq \mathbf{e}_{23}+\mathbf{e}_{11}-\mathbf{e}_{12}-\mathbf{e}_{13}, \quad
\alpha^{13|2}\coloneq\mathbf{e}_{13}+\mathbf{e}_{22}-\mathbf{e}_{12}-\mathbf{e}_{23},  \quad
\alpha^{12|3}\coloneq\mathbf{e}_{12}+\mathbf{e}_{33}-\mathbf{e}_{13}-\mathbf{e}_{23}.
\]
We refer to the bounded ratios $\alpha^{ij|k}$ as the \emph{triangular ratios}. The Alexandrov--Fenchel type ratios can be expressed as the pairwise sums
\[
\alpha^{12}=\alpha^{23|1}+\alpha^{13|2},
 \quad
\alpha^{13}=\alpha^{23|1}+\alpha^{12|3},
\quad
\alpha^{23}=\alpha^{13|2}+\alpha^{12|3},
\]
which are bounded but not primitive on $\L_3^+$.
To see the validity of the displayed inequalities on $\L_3^+$, we use the implication
\[
\text{$
\left(
\begin{matrix}
	p_{11} & p_{12} & p_{13} \\
	p_{12} & p_{22} & p_{23} \\
	p_{13} & p_{23} & p_{33}
\end{matrix}\right)$ is Lorentzian}
\Longrightarrow
\text{$
\left(
\begin{matrix}
	0 & p_{12} & p_{13} \\
	p_{12} & 0 & p_{23} \\
	p_{13} & p_{23} & p_{33}
\end{matrix}\right)$ is Lorentzian.}
\]
The nonnegativity of the determinant of the latter matrix gives the optimal bounding constant $f(\alpha^{12|3})=2$, which is witnessed by the Lorentzian matrices
\[
\left(
\begin{matrix}
	\epsilon & 1 & 1 \\
	1 & \epsilon & 1 \\
	1 & 1 & \frac{2}{1+\epsilon}
\end{matrix}\right) \in \L_3^+ \ \ \text{for sufficiently small $\epsilon >0$.}
\]
Note that $0= \log f(\alpha^{12|3}+\alpha^{13|2})\le \log f(\alpha^{12|3})+\log f(\alpha^{13|2})=2 \log 2$.
In Theorem~\ref{mth:entropy}, we explicitly describe the function $f(\alpha)$ on $\BR(\L_3^+)$.

The inequalities for the primitive bounded ratios $\alpha^{ij|k}$
 appear in the work of Andr\'e Weil on the proof of the Riemann hypothesis for algebraic curves over finite fields, where they are referred to as the \emph{Castelnuovo--Severi inequality} \cite[Theorem 1.5]{Milne}.
This is the first special case of the \emph{reverse Khovanskii--Teissier inequality} \cite[Theorem 5.7]{Lehmann-Xiao}. The inequalities for $\alpha^{ij|k}$ also appear
in early works on Brunn--Minkowski theory; see \cite[Page 396]{Frobenius1915} and \cite[Section~51]{Bonnesen-Fenchel}.\footnote{We thank Ramon van Handel for pointing out these references.}
\end{example}

It turns out that there are no essentially new bounded ratios on $\L_4^+$:
\[
\BR(\L^+_4)=\operatorname{Cone}(\alpha^{ij|k})_{1 \le i,j,k \le 4}, \ \ \text{where $\alpha^{ij|k} \coloneq \mathbf{e}_{ij}+\mathbf{e}_{kk}-\mathbf{e}_{ik}-\mathbf{e}_{jk}$.}
\]
Exactly one new type of bounded ratios appears for $\L_5^+$, which we refer to as the \emph{pentagonal ratios}:
\[
\alpha^{ijk|lm}\coloneq \mathbf{e}_{ij}+\mathbf{e}_{ik}+\mathbf{e}_{jk}+\mathbf{e}_{ll}+\mathbf{e}_{lm}+\mathbf{e}_{mm}-\mathbf{e}_{il}-\mathbf{e}_{jl}-\mathbf{e}_{kl}-\mathbf{e}_{im}-\mathbf{e}_{jm}-\mathbf{e}_{km}.
\]

\renewcommand{\themainthm}{A}
 
\begin{mainthm}
\label{mth:pentagonal}
The \emph{pentagonal inequality} holds for any $(p_{ij}) \in \L_n$:
\[
p_{ij}p_{ik}p_{jk}p_{ll}p_{lm}p_{mm} \le 4 p_{il}p_{jl}p_{kl}p_{im}p_{jm}p_{km} \ \ \text{for any $i,j,k,l,m$.}
\]
The constant $4$ is optimal when all five indices are distinct, and
    \[
\BR(\L_5^+)=\operatorname{Cone}(\alpha^{ijk|lm})_{1 \le i,j,k,l,m \le 5}.
    \]
\end{mainthm} 

Note that the triangular ratios appear as degenerate cases of the pentagonal ratios:
\[
\alpha^{ijk|kk} = \alpha^{ij|k} \ \ \text{for any $i,j,k$.}
\]
The Alexandrov--Fenchel type ratio $\alpha^{ij}$ appears as a degenerate triangular ratio $\alpha^{ii|j}$, which is bounded but not extremal in $\BR(\L_5^+)$.
The cone of bounded ratios $\BR(\L_5^+)$ has forty extremal rays, corresponding to the ten nondegenerate pentagonal ratios and the thirty nondegenerate triangular ratios.
We have a factorization of the pentagonal ratio  into a product of three  triangular ratios and one reciprocal:
\[
\frac{p_{12}p_{13}p_{23}p_{44}p_{45}p_{55}}{p_{14}p_{15}p_{24}p_{25}p_{34}p_{35}} = \frac{p_{12}p_{44}}{p_{14}p_{24}}\cdot\frac{p_{13}p_{55}}{p_{15}p_{35}}\cdot\frac{p_{22}p_{45}}{p_{24}p_{25}}\cdot\left(\frac{p_{22}p_{34}}{p_{23}p_{24}}\right)^{-1}.
\]
 Even though the reciprocal on the right-hand side is unbounded  on its own, the pentagonal ratio on the left-hand side is bounded on $\L_n^+$. Factorizations of this type are not unique. 

\begin{example}
Let $\Delta_1,\Delta_2,\Delta_3$ be equiangular unit line segments in the plane, let $\Delta_4$ be the unit disc,  let $\Delta_5=A$ be an arbitrary planar convex body, and let  $p_{ij}$ be the mixed volume of $\Delta_i$ and $\Delta_j$. In this case, Theorem~\ref{mth:pentagonal} says that 
\[
\operatorname{Area}(A) \ \operatorname{Perimeter}(A)
\le \frac{64}{3\sqrt{3}\,\pi}\,w_1w_2w_3,
\]
where $w_i$ denotes the width of $A$ in the direction perpendicular to $\Delta_i$ for $i=1,2,3$.
\end{example}

In general, we show that $\BR(\L_n^+)$  is a finitely generated polyhedral cone whose extremal rays are generated by the primitive bounded ratios. Up to $S_n$-symmetry, the numbers of primitive bounded ratios for $n=3,4,5,6,7,8$ are $1$, $1$, $2$, $4$, $36$, and $2169$, respectively.
The two new types of primitive bounded ratios for $\L_6^+$ correspond to the inequalities of the form
\[
\frac{p_{11}p_{12}p_{12}p_{13}p_{13}p_{23}p_{44}p_{45}p_{46}p_{55}p_{56}p_{66}}{p_{14}p_{14}p_{15}p_{15}p_{16}p_{16}p_{24}p_{25}p_{26}p_{34}p_{35}p_{36}} \le c, \quad
\frac{p_{11}p_{11}p_{11}p_{12}p_{12}p_{22}p_{34}p_{35}p_{36}p_{45}p_{46}p_{56}}{p_{13}p_{13}p_{14}p_{14}p_{15}p_{15}p_{16}p_{16}p_{23}p_{24}p_{25}p_{26}} \le c,
\]
although we do not know the optimal bounding constants for these inequalities; see Conjecture~\ref{conj:intro-at-most-2}.  
To prove the general statement, we relate $\BR(\L_n^+)$ to the \emph{cut cone}, a finitely generated integral polyhedral cone extensively studied in metric geometry and graph theory. 
For a comprehensive introduction to the geometry and combinatorics of the cut cone, we refer to the book \cite{Deza-Laurent}.

\begin{definition}
For any subset $S \subseteq [n]$, define the \emph{cut vector} $\delta(S)=(\delta(S)_{ij})_{1 \le i < j \le n}$ by
\[
\delta(S)_{ij} = 
\begin{cases}
1, & \text{if $S$ contains exactly one of $i$ and $j$}, \\
0, & \text{otherwise.}
\end{cases}
\]
The \emph{cut cone} $\Cut_n$ is the polyhedral cone in $\RR^{\binom{n}{2}}$ generated by all the cut vectors $\delta(S)$ for $S \subseteq [n]$.
\end{definition}

The extremal rays of $\Cut_n$ are generated by the cut vectors of nonempty proper subsets of $[n]$.
The cut cone can be identified with the space of semimetrics on $n$ points that are isometrically embeddable in $\ell_1$-space \cite[Section 1.1]{Deza-Laurent}. 

We consider the \emph{scaling action} of 
$(\RR_{>0})^n$ on $\Sym_n(\RR_{> 0})$ defined by the formula
\[
	(c_1,\dots, c_n)\cdot (p_{ij}) \coloneqq (c_ic_jp_{ij}).
\]
If $X\subseteq \Sym_n(\RR_{>0})$ is stable under the scaling action, then the orbit space of $X$ can be identified with the set $\underline{X}$ of symmetric matrices in $X$ with diagonal entries $1$.
We say that such matrices in $\underline{X} \subseteq X$ are \emph{normalized}, and define the cone of  \emph{reduced bounded ratios} on $X$ by 
\[
\underline{\BR}(X)
\coloneqq
\left\{ 
(\alpha_{ij})_{1 \le i<j \le n} \;\middle|\; \text{there is $c>0$ such that  $\prod_{1\le i< j\le n}p_{ij}^{\alpha_{ij}}\le c$ for all $(p_{ij}) \in \underline{X}$}   \right\}.
\]
As before, the optimal bounding constants $\underline{f}(\alpha)$ for reduced bounded ratios are defined to be the infimum of all possible constants $c$ satisfying the displayed inequality for all $(p_{ij})$ in $\underline{X}$.
It is straightforward to check that $\underline{\BR}(X)$ is a convex cone in $\RR^{\binom{n}{2}}$ and $\log\underline{f}(\alpha)$ is a convex homogeneous function of degree $1$ on the cone of reduced bounded ratios.

If $\alpha$ is a bounded ratio on $X$, then the corresponding monomial $p^\alpha$
must be invariant under the scaling action of $(\RR_{>0})^n$. 
In other words, if $\alpha\in\BR(X)$, then 
\[
2\alpha_{ii}=-\sum_{j\neq i}\alpha_{ij} \ \ \text{for all $i$}.
\]
Therefore, the projection $\pi: \RR^{\binom{n+1}{2}} \to \RR^{\binom{n}{2}}$ that omits the diagonal entries restricts to a bijection between $\BR(X)$ and $\underline{\BR}(X)$. Moreover,  we have $f(\alpha)=\underline{f}(\pi(\alpha))$ for any bounded ratio $\alpha$.

\renewcommand{\themainthm}{B}
\begin{mainthm}
\label{mth:BR}
	The cone of reduced bounded ratios on $\L_n^+$ is 
     the dual of the cut cone:
	\[
    \underline{\BR}(\L_n^+)=\left\{(\alpha_{ij})_{1 \le i < j\le n} \;\middle|\; \sum_{1 \le i<j \le n} \alpha_{ij} \beta_{ij} \le 0 \ \ \text{for all $(\beta_{ij})_{1 \le i<j \le n} \in \Cut_n$} \right\} \subseteq \RR^{\binom{n}{2}}.
    \]
In particular, the cone of bounded ratios on $\L_n^+$ is a  finitely generated integral polyhedral cone.
\end{mainthm}

Thus, the primitive bounded ratios on $\L_n^+$ are in bijection with the facets of $\Cut_n$. Karp and Papadimitriou showed in  \cite{Karp-Papadimitriou}  that  there is no computationally tractable description of all facets of the cut cone  unless $\mathrm{NP}=\mathrm{coNP}$.
Consequently, obtaining a complete and explicit description of the primitive bounded ratios on $\L_n^+$ is likely to be a challenging problem.

Nevertheless, several interesting classes of facets of $\Cut_n$ are known, the first among which is given by  the \emph{hypermetric inequalities} \cite[Chapter 28]{Deza-Laurent}: For any vector $h=(h_i)_{1 \le i \le n}\in \mathbb{Z}^n$ satisfying $\sum_{i=1}^n h_i=1$, we have
\[
\sum_{1 \le i < j \le n} h_ih_j \beta_{ij} \le 0 \ \ \text{for all $(\beta_{ij})_{1 \le i<j \le n}$ in $\Cut_n$.}
\]
Thus, we may deduce the following statement from Theorem~\ref{mth:BR}.

\begin{corollary}
If  $\sum_{i=1}^n h_i=1$, the vector $(h_ih_j)_{1 \le i<j\le n}$ is a reduced bounded ratio on $\L_n^+$.
\end{corollary}

Bounded ratios of this type are called \emph{hypermetric}. Determining the optimal bounding constants for hypermetric bounded ratios as a function of $h$ is an interesting and challenging problem.
The hypermetric inequality for $h=\mathbf{e}_1+\mathbf{e}_2-\mathbf{e}_3$ gives the triangular ratio $\alpha^{12|3}$, and the hypermetric inequality for $h=\mathbf{e}_1+\mathbf{e}_2+\mathbf{e}_3-\mathbf{e}_4-\mathbf{e}_5$ gives the pentagonal ratio $\alpha^{123|45}$.
Up to $S_n$-symmetry, two additional hypermetric inequalities are needed to describe all facets of $\Cut_6$ \cite[Remark 15.2.11]{Deza-Laurent}. These hypermetric inequalities correspond to 
\[
h=2\mathbf{e}_1+\mathbf{e}_2+\mathbf{e}_3-\mathbf{e}_4-\mathbf{e}_5-\mathbf{e}_6  \ \ \text{and} \ \ 
h=-2\mathbf{e}_1-\mathbf{e}_2+\mathbf{e}_3+\mathbf{e}_4+\mathbf{e}_5+\mathbf{e}_6.
\]
For $n \ge 7$, the cut cone has facets that do not correspond to any hypermetric inequality; see \cite[Section 30.6]{Deza-Laurent} for a complete list of facets of $\Cut_7$. 
For a comprehensive discussion of the facets of the cut cone, we refer the reader to Deza and Laurent’s monograph \cite[Part V]{Deza-Laurent}.

We begin by outlining the proof techniques for our main theorem and situating them within a broader context. Br\"and\'en and the second author introduced \emph{Lorentzian polynomials} \cite{Branden-Huh}, a class of homogeneous polynomials that unify and extend various notions of log-concavity arising in combinatorics and geometry. A quadratic Lorentzian polynomial is, by definition, a quadratic form whose Hessian matrix is Lorentzian. 
 Meanwhile, Baker and Bowler developed a unifying algebraic framework for \emph{matroids over tracts} \cite{Baker-Bowler}, which simultaneously generalized many variants of matroids. Building on this framework, the authors of  \cite{BHKL-triangular} identified the space of Lorentzian polynomials with given support with the corresponding thin Schubert cell in the Grassmannian over the \emph{triangular hyperfield}, up to homeomorphism. This enabled the authors to use Gromov's theorem on $\delta$-hyperbolic spaces in \cite{BHKL-triangular2} to derive new results on Lorentzian polynomials. Our proof is inspired by this perspective, which we now explain in concrete terms.

 Let $\Sym_n(\RR_{ \ge 0})$ denote the set of $n \times n$ symmetric matrices with nonnegative entries.
 
\begin{definition}
Let $\Delta_n(\T_0)$ be the set of matrices $(p_{ij}) \in \Sym_n(\RR_{ \ge 0})$ such that
\[
\text{the maximum among $p_{ij}p_{kl}, p_{ik}p_{jl}, p_{il}p_{jk}$ is achieved at least twice for any $i,j,k,l \in [n]$.}
\]
For a positive real number $p$, let $\Delta_n(\T_p)$ be the set of matrices $(p_{ij}) \in \Sym_n(\RR_{ \ge 0})$ such that
\[
\text{
$(p_{ij}p_{kl})^{1/p} \le
 (p_{ik}p_{jl})^{1/p}+
(p_{il}p_{jk})^{1/p}$  for any $i,j,k,l \in [n]$.}
\]
For any nonnegative real number $p$, let $\Delta^+_n(\T_p)$ denote the intersection $\Delta_n(\T_p) \cap \Sym_n(\RR_{> 0})$.
\end{definition}

Note that $\Delta^+_n(\T_p)$ is invariant under the scaling action of $(\RR_{>0})^n$. As before, we identify the orbit space with the set $\underline{\Delta}^+_n(\T_p)$ of matrices in $\Delta^+_n(\T_p)$ with diagonal entries $1$.
A key observation from \cite{BHKL-triangular} is that 
\[
 \underline{\Delta}_n^+(\T_0)\subseteq  \underline{\L}_n^+\subseteq  \underline{\Delta}_n^+(\T_2),
\]
where $\underline{\L}_n^+$ is the set of $n \times n$  Lorentzian matrices with diagonal entries $1$ and off-diagonal entries positive.
By taking the coordinatewise logarithm and omitting the diagonal coordinates, we have
\[
 \log \underline{\Delta}_n^+(\T_0)\subseteq  \log \underline{\L}_n^+\subseteq  \log \underline{\Delta}_n^+(\T_2) \ \ \text{in} \ \  \RR^{\binom{n}{2}}.
\]
As observed in Section~\ref{sec:hyperbolic}, after adding $2\log 2$ to each off-diagonal entry, $\log \underline{\Delta}_n^+(\T_2)$ parametrizes Gromov's $\delta$-hyperbolic metrics on $n$ points for $\delta= \log 2$, while $ \log \underline{\Delta}_n^+(\T_0)$ parametrizes $0$-hyperbolic metrics, which are precisely the tree metrics. It follows from Gromov's tree approximation theorem \cite[Section 6.1]{Gromov} that the two spaces are within finite Hausdorff distance of one another. Consequently, 
\[
\underline{\BR}(\Delta_n^+(\T_0))=\underline{\BR}(\L_n^+)=\underline{\BR}(\Delta_n^+(\T_2)).
\]
Since $\Cut_n$ is the convex hull of the space of tree metrics on $n$ points, we conclude that $\underline{\BR}(\L_n^+)$ and $\Cut_n$ are dual to each other; see Section~\ref{sec:hyperbolic} for further details.

Having understood the cone of bounded ratios for Lorentzian matrices, it is natural  to ask how the optimal bounding constants $f(\alpha)$ behave as a function on this cone. As mentioned before,
\[
\log f(\alpha+\beta) \le \log f(\alpha)+\log f(\beta) \ \ \text{and} \ \ \log f(c \alpha) = c \log f(\alpha).
\]
In Theorem~\ref{mth:entropy} below, we give an explicit description of the  function $f(\alpha)$ when $n=3$.

Recall that the cone of bounded ratios $\BR(\L_3^+)$ has three extremal rays 
$
\alpha^{23|1}, \alpha^{13|2}, \alpha^{12|3}
$
corresponding to the optimal inequalities
\[
\frac{p_{23}p_{11}}{p_{12}p_{13}} \le 2, \quad \frac{p_{13}p_{22}}{p_{12}p_{23}} \le 2,
\quad \frac{p_{12}p_{33}}{p_{13}p_{23}} \le 2.
\]
Since $\log f$ is a homogeneous function of degree $1$, it is enough to determine 
\[
f(a,b,c)\coloneq f(a \cdot \alpha^{23|1}+b \cdot \alpha^{13|2}+ c \cdot \alpha^{12|3}) \ \ \text{when  $a+b+c=1$ and $a,b,c \ge 0$.}
\]

\renewcommand{\themainthm}{C}
\begin{mainthm}
\label{mth:entropy}
Let $a,b,c$ be nonnegative numbers such that $a+b+c=1$.
\begin{enumerate}[(1)]\itemsep 5pt
\item If $a^2+b^2+c^2-2ab-2ac-2bc\le 0$, then $ f(a,b,c)=1$.
\item If $a^2+b^2+c^2-2ab-2ac-2bc\ge 0$, $a \ge b$,  $a \ge c$, then
\[
f(a,b,c)=2\cdot a^a\cdot b^b \cdot c^c \cdot (2a-1)^{2a-1} \cdot (1-2b)^{2b-1}\cdot (1-2c)^{2c-1}.
\]
\end{enumerate}
See Figure~\ref{fig:triangle} for an illustration.
\end{mainthm}

\begin{figure}[htbp]
    \centering
    \includegraphics[width=0.6\linewidth]{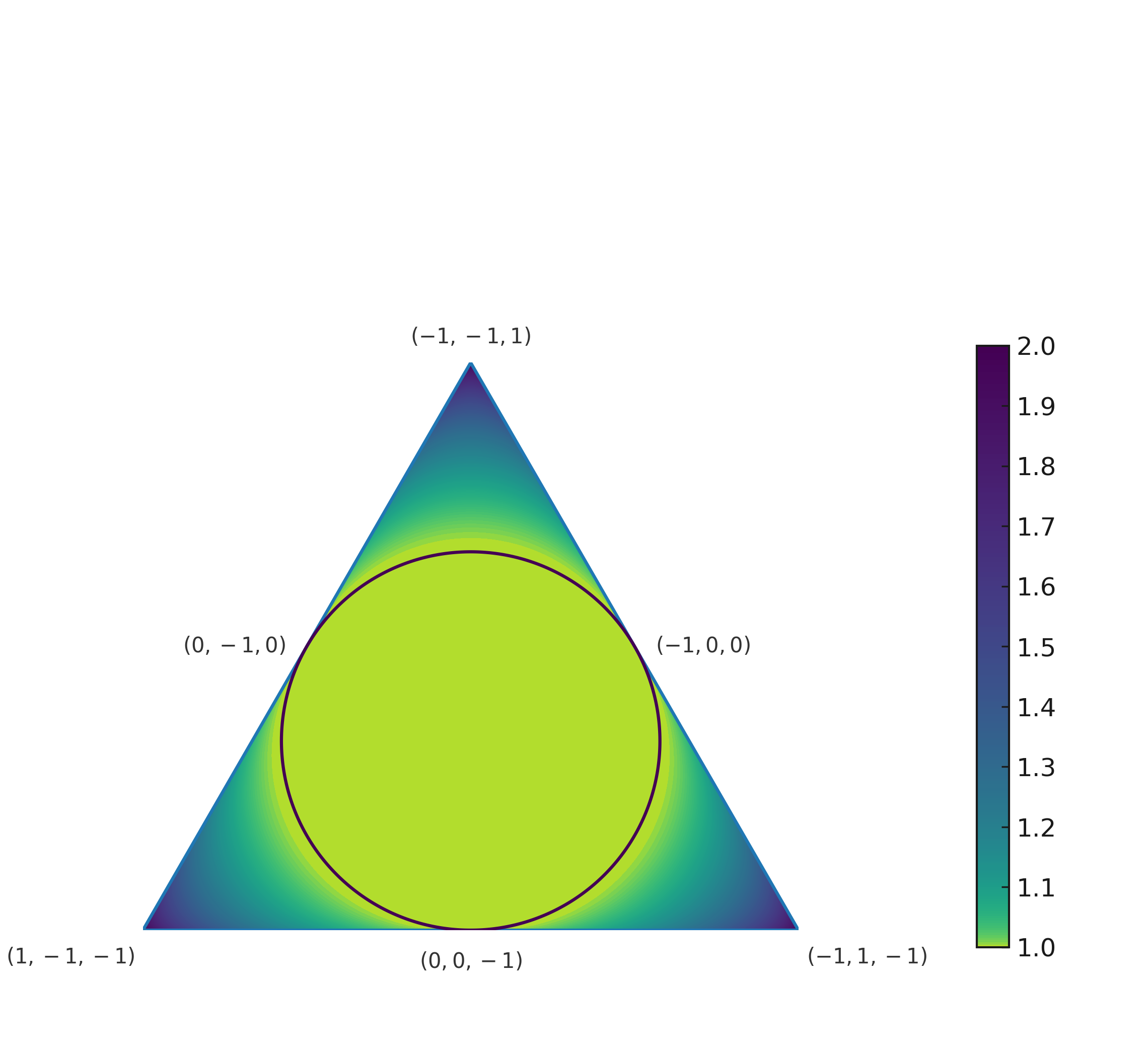}
    \caption{Optimal bounding constants for an equilateral triangle cross section of $\underline{\BR}(\L_3^+)$. The midpoints of the three edges correspond to the Alexandrov--Fenchel type inequalities. Inside the inscribed circle, the optimal bounding constant is $1$. Outside of the inscribed circle, the optimal bounding  constants are given by an entropy-like function.}
    \label{fig:triangle}
\end{figure}

We conclude with a general conjecture on the optimal bounding constant $f(\alpha)$.

\begin{definition}
We say that a reduced bounded ratio on $\L_n^+$ is \emph{normalized} if the sum of its coordinates is $-1$.
\end{definition}

For example, reduced bounded ratios of the form  $\pi(\alpha^{ij|k})=\mathbf{e}_{ij}-\mathbf{e}_{ik}-\mathbf{e}_{jk}$ are normalized. 
Since the set of Lorentzian matrices is preserved when a diagonal entry is replaced by $0$, the coordinate sum of any nonzero reduced bounded ratio is negative. 
Thus, any nonzero reduced bounded ratio is uniquely a positive multiple of a normalized reduced bounded ratio.

\begin{conjecture}
\label{conj:intro-at-most-2}
	The optimal bounding constant for any normalized reduced bounded ratio on $\L_n^+$ is at most 2.
\end{conjecture} 

Using log-convexity of the optimal bounding constants, 
Conjecture~\ref{conj:intro-at-most-2} can be reduced to the case of extremal bounded ratios. Thus,
by Theorem~\ref{mth:pentagonal}, we know that Conjecture \ref{conj:intro-at-most-2} holds for $n\leq 5$. See Section~\ref{sec:rank2} for a closely related conjecture for rank $2$ Lorentzian matrices.\footnote{After the completion of this paper, counterexamples to Conjecture~\ref{conj:intro-at-most-2} were found in \cite{BaldiKummer} and \cite{Zhou}. We do not know if there is a universal constant c such that the optimal bounding constant for any normalized reduced bounded ratio for Lorentzian matrices is at most c. More generally, it would be interesting to determine the optimal bounding constants for bounded ratios of Lorentzian polynomials of fixed degree and number of variables. For results on ternary Lorentzian polynomials, see \cite{BRS,Ternary}.}

\subsection*{Acknowledgements}
The authors gratefully acknowledge the Institute for Advanced Study for providing an inspiring research environment and thank Matt Baker, Chayim Lowen, Ramon van Handel, and Dijia Chen for insightful and stimulating discussions.
Daoji Huang is  supported by the Charles Simonyi Endowment and NSF-DMS2202900.
June Huh was partially supported by the Simons Investigator Grant
MPS-SIM-00825156 and NSF grant DMS-2053308. 
 Botong Wang is partially supported by the NSF grant DMS-1926686.

 \section{The pentagonal inequality}

We prove Theorem~\ref{mth:pentagonal}.
The primary difficulty in proving this theorem lies in determining the optimal bounding constant for the pentagonal inequality. The statement that $\BR(\L_5^+)$ is generated by the pentagonal ratios follows from Theorem~\ref{mth:BR} and the fact that the facets of $\Cut_5$ are defined by the pentagonal ratios \cite[Section 30.6]{Deza-Laurent}.

We begin by observing that the constant $4$ in the pentagonal inequality is, if valid, best possible.
Let $t$ be a nonnegative real number, and consider the matrix
\[
M(t)\coloneq
\begin{pmatrix}
0 & 1 & 1 & t & 2 + t \\[1pt]
1 & 0 & 1 & 2 & 2 \\[1pt]
1 & 1 & 0 & 2 + t & t \\[1pt]
t & 2 & 2 + t & 4t & 4 + 4t \\[1pt]
2 + t & 2 & t & 4 + 4t & 4t
\end{pmatrix}.
\]
It is straightforward to check that $M(t)$ has rank $3$ for $t \ge 0$.
Since the leading $3 \times 3$ principal submatrix of $M$ is nondegenerate and Lorentzian, we see that $M(t)$ is Lorentzian for any $t \ge 0$. We have
\[
\lim_{t \to 0} \frac{p_{12}p_{13}p_{23}p_{44}p_{45}p_{55}}{p_{14}p_{15}p_{24}p_{25}p_{34}p_{35}}=\lim_{t \to 0} \frac{(4t)^2(4+4t)}{4t^2(2+t)^2}=4. 
\]
Since any Lorentzian matrix is a limit of Lorentzian matrices with positive entries \cite[Section~2]{Branden-Huh}, there is a family in $\L_5^+$ on which the evaluation of the pentagonal ratio limits to $4$.

We prepare the proof of the pentagonal inequality with a few auxiliary lemmas.

\begin{lemma}\label{lem:factors}
Suppose that $A$ is a symmetric matrix of the form
\[
A=\begin{pmatrix}
0 & 1 & 1 & p_{14} & p_{15} \\
1 & 0 & 1 & p_{24} & p_{25} \\
1 & 1 & 0 & p_{34} & p_{35} \\
p_{14} & p_{24} & p_{34} & p_{44} & p_{45} \\
p_{15} & p_{25} & p_{35} & p_{45} & p_{55} \\
\end{pmatrix}.
\]
Then there exist unique matrices $B$ and $C$ such that $A=BCB^{T}$, where 
\[
BCB^{T}=\begin{pmatrix}
1 & 0 & 0 & 0 & 0 \\
0 & 1 & 0 & 0 & 0 \\
0 & 0 & 1 & 0 & 0 \\
b_{14} & b_{24} & b_{34} & 1 & 0 \\
b_{15} & b_{25} & b_{35} & 0 & 1 \\
\end{pmatrix}
\begin{pmatrix}
0 & 1 & 1 & 0 & 0 \\
1 & 0 & 1 & 0 & 0 \\
1 & 1 & 0 & 0 & 0\\
0 & 0 & 0 & c_{44} & c_{45} \\
0 & 0 & 0 & c_{45} & c_{55} \\
\end{pmatrix}
\begin{pmatrix}
1 & 0 & 0 & b_{14} & b_{15} \\
0 & 1 & 0 & b_{24} & b_{25} \\
0 & 0 & 1 & b_{34} & b_{35} \\
0 & 0 & 0 & 1 & 0 \\
0 & 0 & 0 & 0 & 1 \\
\end{pmatrix}.
\]
\end{lemma}
\begin{proof}
For any $p_{14}, p_{15}, p_{24}, p_{25}, p_{34}, p_{35}$, there exist unique  $b_{14}, b_{15}, b_{24}, b_{25}, b_{34}, b_{35}$ such that
\[
\begin{pmatrix}
b_{14} & b_{24} & b_{34}\\
b_{15} & b_{25} & b_{35}\\
\end{pmatrix}
\begin{pmatrix}
0 & 1 & 1 \\
1 & 0 & 1 \\
1 & 1 & 0 \\
\end{pmatrix}
=\begin{pmatrix}
p_{14} & p_{24} & p_{34} \\
p_{15} & p_{25} & p_{35} \\
\end{pmatrix}.
\]
This defines the invertible lower triangular matrix $B$.
The matrix $C$ is uniquely determined by the condition $C=B^{-1}A(B^{T})^{-1}$, and it is straightforward to check that $C$ has the required block diagonal shape.
\end{proof}

\begin{lemma}\label{lem:lor_plus_nsd}
Consider the symmetric $2 \times 2$ matrices
\[
M_1=\begin{pmatrix}
    \mu&\nu\\
    \nu&\xi
\end{pmatrix} \ \ \text{and} \ \ 
M_2=\begin{pmatrix}
\alpha&\beta\\
\beta&\gamma
\end{pmatrix}.
\]
If $M_1$ is Lorentzian, $M_2$ is negative semidefinite, and $M_1+M_2$ has nonnegative entries, then
\[
(\mu+\alpha)(\nu+\beta)(\xi+\gamma) \le \mu\nu\xi .
\]
\end{lemma}
\begin{proof}
Since every negative semidefinite matrix is a limit of negative definite matrices and the desired inequality is non-strict, it suffices, by continuity, to consider the case where $M_2$ is negative definite.
Since the claimed inequality is invariant under the symmetric scaling of rows and columns, we may further suppose that $\alpha=\gamma=-1$.

It suffices to show that $(\mu+t\alpha)(\nu+t\beta)(\xi+t\gamma)$ is a decreasing function for $0\leq t\leq 1$. 
Since $M_1+tM_2$ is a Lorentzian matrix for any $0\leq t\leq 1$, it suffices to show that  
\[
\frac{d}{dt}(\mu+t\alpha)(\nu+t\beta)(\xi+t\gamma) \ \bigg|_{t=0}\leq 0.
\]
The assumption that $\alpha=\gamma=-1$ implies $\beta \le 1$.
Hence
\[
\frac{d}{dt}(\mu+t\alpha)(\nu+t\beta)(\xi+t\gamma)\ \bigg|_{t=0}=\mu\nu\gamma+\mu\xi\beta+\nu\xi\alpha
\leq -\mu\nu+\mu\xi-\nu\xi.
\]
Since $M_1$ is Lorentzian, $\mu\xi\leq \nu^2$, so the right-hand side satisfies
\[
-\mu\nu+\mu\xi-\nu\xi \le -2\nu \sqrt{\mu\xi}+\mu\xi \le - \mu\xi \le 0.\qedhere
\]
\end{proof}

\begin{lemma}\label{lem:hard_lemma}
For positive numbers $x_1,x_2,x_3,y_1,y_2,y_3$, set
\begin{align*}
X&\coloneq 2x_1x_2+2x_1x_3+2x_2x_3-x_1^2-x_2^2-x_3^2,\\
Y&\coloneq 2y_1y_2+2y_1y_3+2y_2y_3-y_1^2-y_2^2-y_3^2,\\
Z&\coloneq x_1(y_2+y_3-y_1)+x_2(y_1+y_3-y_2)+x_3(y_1+y_2-y_3).
\end{align*}
If $X,Y,Z$ are nonnegative, then $XYZ < 32x_1x_2x_3y_1y_2y_3$.
\end{lemma}

We note that the final inequality may fail without the assumption that $X,Y,Z$ are nonnegative, for example, when $(x_1,x_2,x_3,y_1,y_2,y_3)=(6,1,1,1,1,6)$.

\begin{proof}
First suppose that $(x_1,x_2,x_3)$ and $(y_1,y_2,y_3)$ satisfy the non-strict triangle inequalities.
Setting $u_i\coloneq x_j+x_k-x_i\geq 0$ and $v_i\coloneq y_j+y_k-y_i\geq 0$, we have
\[
X=u_1u_2+u_1u_3+u_2u_3, \ \ Y=v_1v_2+v_1v_3+v_2v_3, \ \ 2Z=(u_2+u_3)v_1+(u_1+u_3)v_2+(u_1+u_2)v_3.
\]
The goal is to show that
\begin{multline*}
(u_1u_2+u_1u_3+u_2u_3)(v_1v_2+v_1v_3+v_2v_3)((u_2+u_3)v_1+(u_1+u_3)v_2+(u_1+u_2)v_3) \\
< (u_1+u_2)(u_1+u_3)(u_2+u_3)(v_1+v_2)(v_1+v_3)(v_2+v_3).
\end{multline*}
The right-hand side minus the left-hand side can be simplified to
\begin{multline*}
u_1^2 u_2 v_1^2 v_2+ u_1^2 u_3 v_1^2 v_2+ u_1^2 u_2 v_1^2 v_3+ u_1^2 u_3 v_1^2 v_3+ u_1 u_2^2 v_1 v_2^2  + u_2^2 u_3 v_1 v_2^2  + u_1 u_2^2 v_2^2 v_3 + u_2^2 u_3 v_2^2 v_3\\
 + u_1 u_3^2 v_2 v_3^2 + u_2 u_3^2 v_2 v_3^2  + u_1 u_3^2 v_1 v_3^2  + u_2 u_3^2 v_1 v_3^2 - 2 u_1 u_2 u_3 v_1 v_2 v_3.
\end{multline*}
Notice that at most one of the nonnegative numbers $u_1,u_2,u_3$ is zero, and similarly for $v_1,v_2,v_3$. Thus, the sum of the positive coefficient terms is strictly positive. If any of the $u_i$ is zero, then the difference is positive. If $u_1,u_2,u_3$ are all positive, then
the only negative term $- 2 u_1 u_2 u_3 v_1 v_2 v_3 $ can be grouped with the terms $u_1^2 u_2 v_1^2 v_2$ and $u_2 u_3^2 v_2 v_3^2$ to form a square, and the remaining sum is positive.

Thus, without loss of generality, it is sufficient to justify the inequality under the assumption that  $x_1>x_2+x_3$ and $x_2 \ge x_3$. If $y_3<y_2$, then $(y_1+y_3-y_2)<(y_1+y_2-y_3)$. In this case, 
\[
x_2(y_1+y_3-y_2)+x_3(y_1+y_2-y_3)
\leq 
x_2(y_1+y_2-y_3)+x_3(y_1+y_3-y_2)  \ \ \text{and} \ \ 
\]
\[
x_1(y_2+y_3-y_1)+x_2(y_1+y_3-y_2)+x_3(y_1+y_2-y_3)
\leq 
x_1(y_2+y_3-y_1)+x_2(y_1+y_2-y_3)+x_3(y_1+y_3-y_2).
\]
Thus,  swapping $y_2$ and $y_3$ makes $Z$ larger without changing $X$, $Y$, and $x_1x_2x_3y_1y_2y_3$.
Therefore, it is enough to consider the case when $x_1>x_2+x_3$ and $x_2 \ge x_3$ and $y_3 \ge y_2$.

Since
$X= -(x_1-x_2-x_3)^2+4x_2x_3$,
 replacing $x_1$ by $x_2+x_3$ makes $X$ larger without changing $Y$. We check that the same replacement makes the ratio $Z/x_1$ larger:
\begin{multline*}
(y_2+y_3-y_1)+(x_2(y_1+y_3-y_2)+x_3(y_1+y_2-y_3))x_1^{-1}\\
< (y_2+y_3-y_1)+(x_2(y_1+y_3-y_2)+x_3(y_1+y_2-y_3))(x_2+x_3)^{-1}.
\end{multline*}
In fact, we have $x_1>x_2+x_3$ and 
\[
x_2(y_1+y_3-y_2)+x_3(y_1+y_2-y_3)=(x_2+x_3)y_1+(x_2-x_3)(y_3-y_2)\geq (x_2+x_3)y_1> 0.
\]
Since
\[
Z=y_1(x_2+x_3-x_1)+y_2(x_1+x_3-x_2)+y_3(x_1+x_2-x_3),
\]
replacing $x_1$ by $x_2+x_3$ preserves the nonnegativity of $Z$; this reduces the problem to the case when $(x_1,x_2,x_3)$ satisfies the triangle inequalities. Repeating the argument for $(y_1, y_2, y_3)$, we reduce the problem to the case where both $(x_1, x_2, x_3)$ and $(y_1, y_2, y_3)$ satisfy the triangle inequalities, which was considered before.
\end{proof}

\begin{proof}[Proof of Theorem~\ref{mth:pentagonal}]
To show the pentagonal ratio is at most 4, it suffices to consider Lorentzian matrices of the form 
\[
A=\begin{pmatrix}
0 & 1 & 1 & p_{14} & p_{15} \\
1 & 0 & 1 & p_{24} & p_{25} \\
1 & 1 & 0 & p_{34} & p_{35} \\
p_{14} & p_{24} & p_{34} & p_{44} & p_{45} \\
p_{15} & p_{25} & p_{35} & p_{45} & p_{55} \\
\end{pmatrix}.
\]
Indeed, we can set the first three diagonal entries to be zero, since decreasing diagonal entries only decreases eigenvalues. We can then rescale rows and columns using the scaling action to obtain matrix $A$. Both operations do not affect the value of the pentagonal ratio.
We need to show that 
\[
\frac{p_{44}p_{45}p_{55}}{p_{14}p_{24}p_{34}p_{15}p_{25}p_{35}}\leq 4.
\]
By Lemma \ref{lem:factors}, there is a unique factorization 
\[
A=B \cdot C \cdot B^{T}=\begin{pmatrix}
1 & 0 & 0 & 0 & 0 \\
0 & 1 & 0 & 0 & 0 \\
0 & 0 & 1 & 0 & 0 \\
b_{14} & b_{24} & b_{34} & 1 & 0 \\
b_{15} & b_{25} & b_{35} & 0 & 1 \\
\end{pmatrix}
\begin{pmatrix}
0 & 1 & 1 & 0 & 0 \\
1 & 0 & 1 & 0 & 0 \\
1 & 1 & 0 & 0 & 0\\
0 & 0 & 0 & c_{44} & c_{45} \\
0 & 0 & 0 & c_{45} & c_{55} \\
\end{pmatrix}
\begin{pmatrix}
1 & 0 & 0 & b_{14} & b_{15} \\
0 & 1 & 0 & b_{24} & b_{25} \\
0 & 0 & 1 & b_{34} & b_{35} \\
0 & 0 & 0 & 1 & 0 \\
0 & 0 & 0 & 0 & 1 \\
\end{pmatrix}.
\]
Since $A$ is Lorentzian, $C$ has at most one positive eigenvalue. Since the upper-left block of $C$ has exactly one positive eigenvalue, it follows that the lower-right block of $C$ is negative semidefinite.
For $0 \le t \le 1$, we consider the matrices $A(t)$ and $C(t)$ defined by
\[
A(t)=B\cdot C(t)\cdot B^{T}=\begin{pmatrix}
1 & 0 & 0 & 0 & 0 \\
0 & 1 & 0 & 0 & 0 \\
0 & 0 & 1 & 0 & 0 \\
b_{14} & b_{24} & b_{34} & 1 & 0 \\
b_{15} & b_{25} & b_{35} & 0 & 1 \\
\end{pmatrix}
\begin{pmatrix}
0 & 1 & 1 & 0 & 0 \\
1 & 0 & 1 & 0 & 0 \\
1 & 1 & 0 & 0 & 0\\
0 & 0 & 0 & tc_{44} & tc_{45} \\
0 & 0 & 0 & tc_{45} & tc_{55} \\
\end{pmatrix}
\begin{pmatrix}
1 & 0 & 0 & b_{14} & b_{15} \\
0 & 1 & 0 & b_{24} & b_{25} \\
0 & 0 & 1 & b_{34} & b_{35} \\
0 & 0 & 0 & 1 & 0 \\
0 & 0 & 0 & 0 & 1 \\
\end{pmatrix}.
\]
Matrices $B$ and $B^{T}$ correspond to sequences of row and column operations respectively, so 
\[
A(t)=
\begin{pmatrix}
0 & 1 & 1 & p_{14} & p_{15} \\
1 & 0 & 1 & p_{24} & p_{25} \\
1 & 1 & 0 & p_{34} & p_{35} \\
p_{14} & p_{24} & p_{34} & X+tc_{44} & Y+tc_{45} \\
p_{15} & p_{25} & p_{35} & Y+tc_{45} & Z+tc_{55} \\
\end{pmatrix}
\]
for some $X,Y,Z$ independent of $t$. 
The rank of $C(0)$ is $3$, and hence the rank of $A(0)$ is $3$. Thus, all the $4\times 4$ minors of $A(0)$ are zero, which implies that 
\begin{align*}
    X&=\frac{2p_{14}p_{24}+2p_{14}p_{34}+2p_{24}p_{34}-p_{14}^2-p_{24}^2-p_{34}^2}{2},\\
    Z&=\frac{2p_{15}p_{25}+2p_{15}p_{35}+2p_{25}p_{35}-p_{15}^2-p_{25}^2-p_{35}^2}{2}, \\
Y&=\frac{p_{14}p_{25}+p_{14}p_{35}+p_{24}p_{15}+p_{24}p_{35}+p_{34}p_{15}+p_{34}p_{25}-p_{14}p_{15}-p_{24}p_{25}-p_{34}p_{35}}{2}.
\end{align*}
Since $c_{44}$ and $c_{55}$ are nonpositive and $p_{44}=X+c_{44}$ and  $p_{55}=Z+c_{55}$ are nonnegative, 
\[
X+tc_{44} \ge 0 \ \ \text{and} \ \ Z+tc_{55} \ge 0 \ \ \text{for all $0 \le t \le 1$.}
\]
In particular, $X$ and $Z$ are nonnegative.

We now reduce to the case when $Y$ is nonnegative as well.
Since $C(t)$ has at most one positive eigenvalue for all $0 \le t \le 1$,  the same holds for the matrix $A(t)$, and hence
\[
\det \begin{pmatrix}
X+tc_{44} & Y+tc_{45} \\
 Y+tc_{45} & Z+tc_{55} \\    
\end{pmatrix} \le 0.
\]
Suppose $Y$ is negative. Since $Y+c_{45} \ge 0$, we must have $c_{45}>0$ and there is a unique positive $t_1 \le 1$ such that $Y+t_1c_{45}=0$. The determinantal inequality above says that 
\[
X+t_1c_{44}=0  \ \ \text{or} \ \  Z+t_1c_{55}=0.
\]
If $t_1<1$, this implies that $c_{44}=0$ or $c_{55}=0$, and hence $c_{45}=0$ by the negative semidefiniteness of the lower-right block of $C$, reaching a contradiction.
Thus 
\[
p_{45}=Y+c_{45}=0,
\]
in which case the pentagonal inequality holds trivially.

The remaining case to consider is when $X,Y,Z \ge 0$, that is, when the matrix $A(0)$ is Lorentzian. 
We apply Lemma \ref{lem:lor_plus_nsd} to the sum
\[
\begin{pmatrix}
X & Y \\
 Y & Z \\    
\end{pmatrix}+\begin{pmatrix}
c_{44} & c_{45} \\
c_{45} & c_{55} \\    
\end{pmatrix}=\begin{pmatrix}
p_{44} & p_{45} \\
p_{45} & p_{55} \\    
\end{pmatrix},
\]
and conclude that $p_{44}p_{45}p_{55}\leq XYZ$. Thus, the pentagonal ratio of $A(0)$ is greater than or equal to the pentagonal ratio of $A$. Thus, it is enough to show the pentagonal inequality for $A(0)$, which reads
\[
XYZ\leq 4p_{14}p_{24}p_{34}p_{15}p_{25}p_{35}.
\]
This is precisely the content of Lemma \ref{lem:hard_lemma}.
\end{proof}

It is interesting to compare the above analysis with that for the pentagonal ratio on $\Delta_n^+(\T_p)$, which is much simpler.

\begin{proposition}
For any nonnegative $p$ and any matrix $(p_{ij}) \in \Delta_5^+(\T_p)$, we have
\[
\frac{p_{12}p_{13}p_{23}p_{45}p_{44}p_{55}}{p_{14}p_{15}p_{24}p_{25}p_{34}p_{35}} \le 8^{p}.
\]
The equality is achieved by the matrix 
\[
\begin{pmatrix}
1 & 2^{\frac{p}{2}} & 2^{\frac{p}{2}} & 1 & 1 \\
2^{\frac{p}{2}} & 1 & 2^{\frac{p}{2}} & 1 & 1 \\
2^{\frac{p}{2}} & 2^{\frac{p}{2}} & 1 & 1 & 1 \\
1 & 1 & 1 & 2^{\frac{p}{2}} & 2^{\frac{p}{2}} \\
1 & 1 & 1 & 2^{\frac{p}{2}} & 2^{\frac{p}{2}}
\end{pmatrix} \in \Delta^+_5(\T_p).
\]
\end{proposition}
\begin{proof}
We first consider the case when $p$ is positive. We have
\[
(p_{44}p_{13})^{1/p} \le 2 (p_{14}p_{34})^{1/p}, ~~ (p_{55}p_{23})^{1/p} \le 2 (p_{25}p_{35})^{1/p}, \ \ \text{so} \ \ (p_{44}p_{55}p_{13}p_{23})^{1/p} \le 4 (p_{14}p_{34}p_{25}p_{35})^{1/p}.
\]
Similarly, we have
\[
(p_{44}p_{23})^{1/p} \le 2 (p_{24}p_{34})^{1/p}, ~~ (p_{55}p_{13})^{1/p} \le 2 (p_{15}p_{35})^{1/p}, \ \ \text{so} \ \ (p_{44}p_{55}p_{13}p_{23})^{1/p} \le 4 (p_{24}p_{34}p_{15}p_{35})^{1/p}.
\]
Using the triangle inequality $(p_{12}p_{45})^{1/p} \le (p_{14}p_{25})^{1/p}+(p_{15}p_{24})^{1/p}$, we get
\begin{align*}
(p_{12}p_{13}p_{23}p_{44}p_{55}p_{45})^{1/p} &\le (p_{44}p_{55}p_{13}p_{23}p_{14}p_{25})^{1/p}+(p_{44}p_{55}p_{13}p_{23}p_{15}p_{24})^{1/p} \\
&\le 8(p_{14}p_{15}p_{24}p_{25}p_{34}p_{35})^{1/p}.
\end{align*}
Since $\Delta_n(\T_0) \subseteq \Delta_n(\T_p)$ for all $p>0$, the case of $p=0$ follows from taking the limit $p \to 0$.
\end{proof}

\section{Lorentzian matrices and $\delta$-hyperbolic spaces}\label{sec:hyperbolic}

We give a detailed proof of Theorem~\ref{mth:BR} outlined in the introduction.

\begin{proposition}
\label{prop:inclusions}
	For any $n>0$, we have 
    $\Delta_n(\T_0)\subseteq \L_n\subseteq \Delta_n(\T_2)$.
\end{proposition}

In the language of \cite{BHKL-polymatroids}, the proposition states that every rank $2$ polymatroid over $\T_0$ is Lorentzian, and every Lorentzian matrix is a rank $2$ polymatroid over $\T_2$.

\begin{proof}
The first inclusion is a special case of  \cite[Corollary 3.16]{Branden-Huh} for quadratic polynomials. The second inclusion follows from \cite{BHKL-triangular2}, whose proof we reproduce here.

Suppose $(p_{ij})$ is a Lorentzian matrix. If $i,j,k,l$ are distinct indices, we consider the matrix 
\[
\begin{pmatrix}
0 & p_{ij} & p_{ik} & p_{il} \\
p_{ij} & 0 & p_{jk} & p_{jl}  \\
p_{ik} & p_{jk} & 0 & p_{kl}  \\
p_{il} & p_{jl} & p_{kl} & 0
\end{pmatrix} 
\]
obtained from replacing the diagonal entries of a principal submatrix of $(p_{ij})$ by zero. This matrix is Lorentzian with 
\begin{multline*}
\det=-\left(\sqrt{p_{ij}p_{kl}}+\sqrt{p_{ik}p_{jl}}+\sqrt{p_{il}p_{jk}}\right)\cdot(-\sqrt{p_{ij}p_{kl}}+\sqrt{p_{ik}p_{jl}}+\sqrt{p_{il}p_{jk}})\\
      \cdot(\sqrt{p_{ij}p_{kl}}-\sqrt{p_{ik}p_{jl}}+\sqrt{p_{il}p_{jk}})\cdot(\sqrt{p_{ij}p_{kl}}+\sqrt{p_{ik}p_{jl}}-\sqrt{p_{il}p_{jk}}) \le 0,
\end{multline*}
so $\sqrt{p_{ij}p_{kl}}$, $\sqrt{p_{ik}p_{jl}}$, $\sqrt{p_{il}p_{jk}}$ form three sides of a triangle \cite[Proposition 3.1]{HHMWW}.

If there are exactly three distinct indices among $i,j,k,l$, we may suppose $k=l$. 
We need to show that $\sqrt{p_{ij}p_{kk}}$, $\sqrt{p_{ik}p_{jk}}$,
 $\sqrt{p_{ik}p_{jk}}$ form sides of a triangle, that is, $\sqrt{p_{ij}p_{kk}} \le 2\sqrt{p_{ik}p_{jk}}$.
This follows from the Lorentzian condition for $(p_{ij})$ because
\[
\det\begin{pmatrix}
	0 & p_{ij} & p_{ik} \\
	p_{ij} & 0      & p_{jk} \\
	p_{ik} & p_{jk} & p_{kk}
\end{pmatrix}=p_{ij}(2p_{ik}p_{jk}-p_{ij}p_{kk})\ge 0.
\]
If three of $i,j,k,l$ are equal, then the statement is trivial. Thus, the only remaining nontrivial case is when $i=k$ and $j=l$. We need to check that $\sqrt{p_{ii}p_{jj}} \le 2p_{ij}$, and this follows from
\[
\det\begin{pmatrix}
	p_{ii} & p_{ij}  \\
	p_{ij} & p_{jj}      
\end{pmatrix}=p_{ii}p_{jj}-p_{ij}^2\le 0. \qedhere
\]
\end{proof}

We now recall a tree metric in terms of phylogenetic trees, following \cite[Section 4.3]{Maclagan-Sturmfels}. A \emph{phylogenetic} tree on $n$ leaves is a tree with $n$ labeled leaves and no vertices of degree 2.
We say that $(d_{ij})_{1 \le i< j \le n}$ is a \emph{tree metric} if there is a phylogenetic tree $\tau$ with $n$ leaves and nonnegative edge lengths $\ell$ on $\tau$ such that
\[
\mathrm{d}_{ij}=\bigl( \text{the sum of all $\ell(e)$ over all edges $e$ in the unique path from $i$ to $j$ in $\tau$}\bigr).\footnote{In \cite[Section 4.3]{Maclagan-Sturmfels}, edge lengths are allowed to be non-positive. A vector $d$ is a tree metric in this sense if and only if the sum of $d$ and a constant vector of large positive value is a tree metric in the more restrictive sense above.}
\]
A tree metric uniquely determines a phylogenetic tree together with the assignment of edge lengths, up to label-preserving isomorphism and the contraction of zero-length internal edges.
\begin{lemma}
\label{lem:convhullT0} 
The following statements hold for $\underline{\Delta}_n^+(\T_0)$.
\leavevmode
\begin{enumerate}[(1)]\itemsep 5pt
    \item The set of coordinatewise logarithms $\log \underline{\Delta}_n^+(\T_0)\subseteq \RR^{\binom{n}{2}}$ is equal to the space of tree metrics on $n$ leaves. 
    \item 	The convex hull of $\log \underline{\Delta}_n^+(\T_0)$ in $\RR^{\binom{n}{2}}$
    is equal to the cut cone $\Cut_n$.
    \item The cone of reduced bounded ratios $\underline{\BR}(\Delta_n^+(\T_0))$ is the dual of the cut cone $\Cut_n$.
\end{enumerate}
\end{lemma}

\begin{proof}
Note that the points in  $\log \underline{\Delta}_n^+(\T_0)$ are of the form $(q_{ij})=(\log p_{ij})$, where 
\[
q_{ii}=\log p_{ii}=0 \ \ \text{and} \ \ q_{ij} =\log p_{ij} \ge \log \sqrt{p_{ii}p_{jj}} =0,
\]
and in addition, for distinct indices $i,j,k,l$, we have
\[
q_{ij}+q_{kl}=\log p_{ij}p_{kl} \le \log \max(p_{ik}p_{jl},p_{il}p_{jk}) = \max(q_{ik}+q_{jl}, q_{il}+q_{jk}).
\]
Thus,  the first statement is the characterization of tree metrics by the ``four-point condition'', which is exactly the logarithm of the condition defining $\underline{\Delta}^+_n(\T_0)$ \cite[Lemma 4.3.6]{Maclagan-Sturmfels}.

For the second statement, observe that the cut vector
 $\delta(S)$ is the tree metric corresponding to the phylogenetic tree that consists of a single edge of length $1$ that connects the leaves in $S$ with the leaves not in $S$. 
Therefore, the extremal rays of $\Cut_n$ are in $\log \underline{\Delta}_n^+(\T_0)$. 
For the other inclusion, consider a tree metric $d$  given by a phylogenetic tree $\tau$ and a length function $\ell$. For each edge $e$ of $\tau$, let $S(e)$ be the set of leaves that remain connected to the distinguished leaf $1$ after removing $e$ from $\tau$. Then $d$ admits the decomposition
\[
d=\sum_e \ell(e) \delta(S(e)),
\]
showing that $d$ is in the cut cone.
This proves the second statement.

For the third statement, notice that
\begin{align*}
\underline{\BR}\left(\Delta_n^+(\T_0)\right)
&=\left\{ \alpha \in \RR^{\binom{n}{2}}  \;\middle|\;  \exists c> 0\text{ such that } \sum_{1\le i< j\le n}\alpha_{ij}q_{ij}\le 	\log c \ \ \text{for all $q\in \log \underline{\Delta}_n^+(\T_0)$}\right\}\\
&=\left\{ \alpha\in \RR^{\binom{n}{2}}  \;\middle|\;  \exists c> 0\text{ such that } \sum_{1\le i< j\le n}\alpha_{ij}q_{ij}\le 	\log c \ \ \text{for all $q\in \Cut_n$}\right\}\\
&= \left\{ \alpha \in \RR^{\binom{n}{2}}  \;\middle|\;  \sum_{1\le i< j\le n}\alpha_{ij}q_{ij}\le 0 \ \ \text{for all $q\in \Cut_n$}\right\}.
\end{align*}
The first equality is the definition of $\underline{\BR}$, the second follows from the fact that $\Cut_n$ is the convex hull of $\log \underline{\Delta}_n^+(\T_0)$, and the third follows from the fact that $\Cut_n$ is a cone.
\end{proof}

We now review Gromov's tree approximation theorem for $\delta$-hyperbolic spaces. 
Let $\delta$ be a nonnegative real number. A metric $d$  on $[n]$ 
is called \emph{$\delta$-hyperbolic} if
\[
d_{ij}+d_{kl}\le \max \bigl(d_{ik}+d_{jl}+2\delta, d_{il}+d_{jk}+2\delta\bigr) \ \ \text{for all $i,j,k,l\in [n]$.}
\]
As noted above, when $\delta=0$, the condition is equivalent to $d$ being a tree metric. 
We will use the following corollary of the tree approximation theorem; see {\cite[Section 6.1]{Gromov} and \cite[Chapitre 2]{GdlH}}.

\begin{proposition}
	\label{prop:tree-approx}
	For any positive real number $\delta$ and a positive integer $n$,  
there is a constant $c$ 
such that any $\delta$-hyperbolic metric $d$ on $[n]$ can be approximated by a $0$-hyperbolic metric $d'$: 
		\[
        d_{ij}-c\le d'_{ij}\le d_{ij} \ \ \text{for all $i,j \in [n]$}.
        \]
\end{proposition}

\begin{corollary}
\label{lem:finite-hausdorff}
For any positive integer $n$, there is a constant $e^c$ such that
	for any $(p_{ij})\in \underline{\Delta}_n^+(\T_2)$, there is $(p_{ij}')\in \underline{\Delta}_n^+(\T_0)$ satisfying the condition
	\[e^{-c}p_{ij}\le p_{ij}' \le e^{c} p_{ij}\ \ \text{for all }i,j\in [n].\]
\end{corollary}
\begin{proof}
	We reformulate the statement in terms of log coordinates, where $q_{ij}=\log p_{ij}$. Since $(p_{ij})\in \underline{\Delta}_n^+(\T_2)$, we have 
	\begin{enumerate}[(1)]\itemsep 5pt
	\item $q_{ii}=0$ for all $i\in [n]$;
\item $q_{ij}\leq q_{ik}+q_{jk}+2\log 2$ for all distinct $i,j,k\in [n]$;
\item $q_{ij}+q_{kl}\leq \max(q_{ik}+q_{jl}+2\log2, q_{il}+q_{jk}+2\log 2)$ for all distinct $i,j,k,l\in [n]$,
	\end{enumerate}
	where the last condition follows from 
    \[
    \sqrt{p_{ij}p_{kl}}\leq \sqrt{p_{ik}p_{jl}}+\sqrt{p_{il}p_{jk}}\leq \max(2\sqrt{p_{ik}p_{jl}}, 2\sqrt{p_{il}p_{jk}}).
    \]
Now, for all $i, j\in [n]$, define
\[
d_{ij}=\begin{cases}
    0\quad &\text{when $i=j$,}\\
    q_{ij}+2\log 2 \quad &\text{when $i\neq j$.}
\end{cases}
\]
Then, $d_{ij}$
is a $\delta$-hyperbolic metric on $[n]$ for $\delta=\log 2$. 
By Proposition~\ref{prop:tree-approx}, there is a $0$-hyperbolic metric $d'_{ij}$ on $[n]$ and a constant $c'$ such that $d_{ij}-c'\leq d'_{ij}\leq d_{ij}$ for all $i,j\in [n]$. Thus,
\[
q_{ij}-c\leq d'_{ij}\leq q_{ij}+c 
\]
for $c=\max\{c', 2\log 2\}$ and any $i,j\in [n]$.
Let $p'_{ij}=e^{d'_{ij}}$. Then, 
\[
e^{-c}p_{ij}\le p_{ij}' \le e^{c} p_{ij}\ \ \text{for all }i,j\in [n].
\]
Moreover, by Lemma~\ref{lem:convhullT0}, $(p_{ij}')\in \underline{\Delta}_n^+(\T_0)$. 
\end{proof}

For the following lemma, let $X\subseteq Y$ be arbitrary subsets of $n \times n$ symmetric matrices with positive entries.

\begin{lemma}
\label{lem:same-BR}
Suppose that there is a constant $c$ such that, 
for any $(p_{ij}) \in Y$, there is 
$(p'_{ij})\in X$ satisfying $e^{-c}p_{ij}\le p'_{ij}\le e^{c} p_{ij}$ for all $i,j \in [n]$. Then $\BR(X)=\BR(Y)$.
\end{lemma}
\begin{proof}
Since $X\subseteq Y$, we have $\BR(Y)\subseteq \BR(X)$. 
We prove the reverse inclusion. Given $(p_{ij}) \in Y$,
our assumption says that there is $(p_{ij}') \in X$ such that
\[
e^{-c}p_{ij}\le p'_{ij}\le e^{c} p_{ij} \ \ \text{for all $i,j \in [n]$.}
\]
If $\prod_{i\le  j}  {p_{ij}'}^{\alpha_{ij}}$ is bounded above by a constant, then the same is true for $\prod_{i\le  j}  {p_{ij}}^{\alpha_{ij}}$.
\end{proof}

\begin{proof}[Proof of Theorem~\ref{mth:BR}]
By Proposition~\ref{prop:inclusions}, Corollary~\ref{lem:finite-hausdorff}, and 
Lemma~\ref{lem:same-BR}, we have
\[
\underline{\BR}(\Delta_n^+(\T_0))=\underline{\BR}(\L_n^+)=\underline{\BR}(\Delta_n^+(\T_2)).
\]
Lemma~\ref{lem:convhullT0} says that this cone is dual to the cut cone $\Cut_n$.
\end{proof}

\section{Optimal bounding constants for $\L_3^+$}

We determine the optimal bounding constants on $\underline{\BR}(\L_3^+)$.

\begin{proof}[Proof of Theorem~\ref{mth:entropy}]
We consider the triangular section of the cone of reduced bounded ratios 
\[
\left\{a \cdot \alpha^{23|1}+b \cdot \alpha^{13|2}+c \cdot \alpha^{12|3}  \;\middle|\;  \text{$a,b,c$ are nonnegative numbers that sum to $1$}\right\} \subseteq \underline{\BR}(\L_3^+).
\]
The condition $a^2+b^2+c^2-2ab-2ac-2bc=0$ defines an inscribed circle in the given triangle, where the three tangent points are given by
\[
(a,b,c) \ = \ (1/2,1/2,0), \ \ (1/2,0,1/2), \ \ (0,1/2,1/2).
\]
These points are the midpoints of the sides of the triangle, corresponding to the Alexandrov--Fenchel inequalities. 

The inscribed circle splits the triangle into four regions. The part within the circle is defined by 
\[
a^2+b^2+c^2-2ab-2ac-2bc\leq 0.
\]
Outside the circle, the connected component containing the vertex $\alpha^{23|1}$ is defined by
\[
a^2+b^2+c^2-2ab-2ac-2bc > 0,  \ \ a > b+c.
\]
The other two regions outside of the inscribed circle are described similarly. 

We may reformulate the objective as the problem of   finding the supremum of 
\[
R(p_{12},p_{13},p_{23})\coloneqq(p_{12})^{c-b-a}(p_{13})^{b-c-a}(p_{23})^{a-b-c},
\]
subject to certain constraints on $p_{12},p_{13},p_{23}$. More precisely, we want to find the supremum of $R(p_{12},p_{13},p_{23})$ in terms of the parameters $a,b,c$  when $p_{12},p_{13},p_{23}\geq 1$ and
\[
\det\begin{pmatrix}
	1 & p_{12} & p_{13} \\
	p_{12} & 1      & p_{23} \\
	p_{13} & p_{23} & 1
\end{pmatrix}= 
1+2p_{12}p_{13}p_{23}-p^2_{12}-p^2_{13}-p^2_{23}\geq 0.
\]
These inequalities are equivalent to the condition that  the $3 \times 3$ matrix $(p_{ij})$ is Lorentzian.

We apply Lagrange's method to $R(p_{12},p_{13},p_{23})$ separately on the three regions outside the inscribed circle. On the region containing $\alpha^{23|1}$,  one finds that the interior of the domain has no critical points, and thus the maximum is attained on the boundary. 
The interesting case is when 
\[
1+2p_{12}p_{13}p_{23}-p^2_{12}-p^2_{13}-p^2_{23}=0.
\]
Since $p_{12},p_{13},p_{23} \ge 1$, the condition is equivalent to 
\[
p_{23}=p_{12}p_{13}\pm\sqrt{(p^2_{12}-1)(p^2_{13}-1)}.
\]
Since $a>b+c$, the function $R(p_{12},p_{13},p_{23})$ increases in $p_{23}$, and we choose the upper branch for the maximization.
Thus, setting $x=p^{-2}_{12}$ and $y=p^{-2}_{13}$, the question reduces to finding the supremum of 
\[
r(x,y) \coloneq x^{b}y^{c}\left(1+\sqrt{(1-x)(1-y)}\right)^{a-b-c} \ \ \text{when $0< x,y \leq 1$}.
\] 
The above problem is equivalent to the search for the  global maximum of $r(x,y)$ extended continuously to $0 \leq x,y \leq 1$. On the boundary of the square $0 \leq x,y \leq 1$, the function $r(x,y)$ is at most 1, which is no larger than the value of $r(x,y)$ at its interior critical point. 
The unique critical point of $r(x,y)$ has coordinates 
\[
x_0=4ab/(a+b-c)^2  \ \ \text{and} \ \  y_0=4ac/(a-b+c)^2,
\]
and they satisfy $0 \le x_0, y_0 \le 1$ because $a^2+b^2+c^2-2ab-2ac-2bc >0$. Evaluating, we get
\[
r(x_0,y_0)=2\cdot a^a\cdot b^b \cdot c^c \cdot (2a-1)^{2a-1} \cdot (1-2b)^{2b-1}\cdot (1-2c)^{2c-1}.
\]
When $a^2+b^2+c^2-2ab-2ac-2bc=0$, we have $(a+b-c)^2=4ab$ and $(a-b+c)^2=4ac$. In this case, the critical point is $(1,1)$ and $r(x,y)$ has value $1$ at the critical point. 
Thus, when $(a,b,c)$ lies on the boundary of the inscribed circle, $f(a,b,c)=1$. By the log-convexity of $f(a,b,c)$, 
 we conclude that $f(a,b,c)\leq 1$ inside the circle. On the other hand, the matrix with all entries $1$ is Lorentzian, and any ratio for this matrix equals $1$. Thus, $f(a,b,c)=1$ inside the circle.
\end{proof}

We compare the optimal bounding constants $f(a,b,c)$ for $\L_3^+$ with the corresponding constants $f_p(a,b,c)$ for $\Delta_3^+(\T_p)$, which is piecewise log-linear. 
\begin{proposition}
Consider the reduced bounded ratio 
$\alpha=a \cdot \alpha^{23|1}+b \cdot \alpha^{13|2}+c \cdot \alpha^{12|3}$ on $\Delta_3^+(\T_p)$, where $a,b,c$ are nonnegative numbers. 
The optimal bounding constant for $\alpha$ is given by
\[
f_p(a,b,c)=\begin{cases}
2^{ap} & \text{if} \; a>b+c,\\
2^{bp} & \text{if} \; b>a+c,\\
2^{cp} & \text{if} \; c>a+b,\\
2^{\frac{p(a+b+c)}{2}} & \text{otherwise}.
\end{cases}
\]	    
\end{proposition}
\begin{proof}
Let $q_{ij}=\log_2(p_{ij}),$ where $i,j$ are any two distinct indices from $\{1,2,3\}$. In log coordinates, the problem reduces to the maximization of the linear function
\[
\ell(a,b,c)\coloneq a(q_{23}-q_{12}-q_{13})+b(q_{13}-q_{12}-q_{23})+c(q_{12}-q_{13}-q_{23}), 
\]
where the domain is given by the linear inequalities 
\begin{align*}
q_{23} &\leq p + q_{12} + q_{13}, & \
q_{13} &\leq p + q_{12} + q_{23}, & \
q_{12} &\leq p + q_{13} + q_{23}, \\
0 &\leq p + q_{12}+q_{12}, & \
0 &\leq p + q_{13}+q_{13}, & \
0 &\leq p + q_{23}+q_{23}.
\end{align*}
The maximum of $\ell(a,b,c)$ is attained at one of the vertices of the domain. The solution to the system contains four vertices
\[
 (0,-p/2,-p/2), \ \ (-p/2,0,-p/2), \ \ (-p/2,-p/2,0), \ \ \text{and} \ \ (-p/2,-p/2,-p/2),
\]
corresponding to the four cases in the displayed equation for $f_p(a,b,c)$.
\end{proof}

\section{Rank 2 Lorentzian matrices}\label{sec:rank2}

We show that the cone of bounded ratios on Lorentzian matrices is equal to the cone of bounded ratios on rank $2$ Lorentzian matrices. 
The space of rank $2$ Lorentzian matrices admits a natural parametrization, which allows us 
 to formulate a stronger notion of positivity called \emph{subtraction-freeness}. Throughout, we work with an $n \times n$ symmetric matrix $M$ with real entries. 

\begin{lemma}\label{lem:quadr1} 
The quadratic form $Q(x)=x^{T}Mx$ factors into a product of two linear forms if and only if there is a rank at most $1$ matrix $B$  such that $M=\frac{1}{2}(B+B^{T})$.    
\end{lemma}
\begin{proof}
If the quadratic form $Q(x)$ factors into linear forms, we can rewrite
\[
Q(x)=x^{T}Mx=(x^{T}a)(b^{T}x)=x^{T}(ab^{T})x=x^T(ba^T)x,
\]
where $a$ and $b$ are column vectors. Then  $B=ab^{T}$ is a rank at most $1$ matrix and $Q(x)=x^{T}\left(\frac{B+B^{T}}{2}\right)x$.
Conversely, if $M=\frac{1}{2}(B+B^{T})$ for a rank at most $1$ matrix $B$, then 
\[Q(x)=x^{T}Mx=x^{T}\left(\frac{B+B^{T}}{2}\right)x=x^{T}Bx.\]
Since $B$ has rank at most $1$, there are real vectors $a$ and $b$ such that
$x^{T}Bx=x^{T}(ab^{T})x=(x^{T}a)(b^{T}x)$. 
\end{proof}

\begin{lemma}\label{lem:rank2} 
 There is a rank at most $1$ matrix $B$ such that $M=\frac{1}{2}(B+B^{T})$ if and only if $M$ has rank at most $2$ and all $2\times2$ principal minors of $M$ are nonpositive.
\end{lemma}
\begin{proof}
If  $M=\frac{1}{2}(B+B^{T})$ for a rank at most $1$ matrix $B$, then, by Lemma \ref{lem:quadr1}, there is a factorization 
\[
x^{T}Mx=\left(\sum_{1\leq i\leq n}a_ix_i\right)\left(\sum_{1\leq i\leq n}b_ix_i\right).
\]
Thus, $M$ has rank at most $2$ and any $2 \times 2$ principal minor of $M$ is nonpositive:
\[(2a_ib_i)(2a_jb_j)-(a_ib_j+a_jb_i)^2=-a_i^2b_j^2+2a_ia_jb_ib_j-a_j^2b_i^2 \leq 0.\]

For the converse, suppose that $M$ has rank at most $2$ and that all $2\times2$ principal minors of $M$ are nonpositive. We want to find a matrix $B=ab^T$ 
for some real vectors $a,b$ such that $M=\frac{1}{2}(B+B^{T})$. If $M$ is rank at most 1, $M=uu^T$ or $-uu^T$ for some vector $u$. Accordingly, we may take $a=b=u$ in the first case and $a=u$, $b=-u$ in the second. Now suppose $M$ is rank 2. 
Since $M$ has all $2\times 2$ principal minors nonpositive, $M$ cannot be positive or negative semidefinite, so $M$ must have eigenvalues $\lambda_1>0>\lambda_2$. 
Suppose $u$ and $v$ are orthonormal eigenvectors with these eigenvalues. Then $M=\lambda_1 uu^T+\lambda_2vv^T$. 
It then follows that we can take $a=(\lambda_1)^{1/2}u+(-\lambda_2)^{1/2}v$ and $b= (\lambda_1)^{1/2}u-(-\lambda_2)^{1/2}v$.
\end{proof}

\begin{proposition}
For any nonnegative vectors $(a_i)_{1 \le i \le n}, (b_i)_{1 \le i \le n}$, the Hessian of the quadratic form $\left(\sum_{1\leq i\leq n}a_ix_i\right)\left(\sum_{1\leq i\leq n}b_ix_i\right)$ is Lorentzian, and every Lorentzian matrix of rank at most $2$ is of this form.
\end{proposition}

\begin{proof}
Let $M$ be the Hessian of the quadratic form  $Q=\left(\sum_{1\leq i\leq n}a_ix_i\right)\left(\sum_{1\leq i\leq n}b_ix_i\right)$, where $a,b$ are nonnegative real vectors. 
By Lemma \ref{lem:quadr1} and Lemma \ref{lem:rank2}, we know that $M$ is of rank at most $2$ and all $2\times2$ principal minors of $M$ are nonpositive. Since $a,b$ are nonnegative, the entries of $M$ are nonnegative. Thus, the symmetric matrix $M$ is a Lorentzian matrix of rank at most $2$.

Conversely, for any nonzero Lorentzian matrix $M$ of rank at most $2$, by Lemma \ref{lem:quadr1} and Lemma \ref{lem:rank2}, the corresponding quadratic form $x^{T}Mx$ factors into a product 
\[x^{T}Mx=\left(\sum_{1\leq i\leq n}a_ix_i\right)\left(\sum_{1\leq i\leq n}b_ix_i\right) \ \ \text{for some real vectors $a,b$.}\]  
It remains to show that vectors $a,b$ can be chosen nonnegative. Since the entries of $M$ are nonnegative, $a_ib_i\geq 0$ for all $1\leq i\leq n$. 

First, consider the case when $a_ib_i=0$ for $1\leq i\leq n$. Since $M$ is nonzero, replacing all $a_i$ and $b_i$ by their negatives and permuting the indices if necessary, we may suppose that $a_1>0$, which in turn implies that $b_1=0$. For any $b_i\neq 0$, since $a_1b_i+a_ib_1\geq0$, we have $b_i>0$. Thus, all nonzero $b_i$ are positive. Since $M$ is nonzero, one of the $b_i$ must be positive. Then we apply the same argument to conclude that all nonzero $a_i$ are positive. 

When $a_ib_i> 0$ for some $i$, without loss of generality, we may assume that $a_1>0$ and $b_1>0$. Since $a_1b_i+a_ib_1$ is nonnegative as well as $a_ib_i\geq0$ for any $i$, we conclude that all $a_i$ and $b_i$ are nonnegative. 
\end{proof}

Let $\L_n^+(2)$ denote the set of $n\times n$ Lorentzian matrices of rank at most $2$ with positive entries.
Since $\L_n^+(2)\subseteq \L_n^+$, we have 
$\underline{\BR}\bigl(\L_n^+\bigr) \subseteq \underline{\BR}\bigl(\L_n^+(2)\bigr)$.
We show that these cones coincide and they are dual to the cut cone $\Cut_n$.

\begin{proposition}\label{thm:BR_rk2}
We have 
	$
    \underline{\BR}\bigl(\L_n^+\bigr)=
    \underline{\BR}\bigl(\L_n^+(2)\bigr)$.
\end{proposition}

\begin{proof}
Theorem~\ref{mth:BR} says that $\Cut_n$ is dual to  $\underline{\BR}\bigl(\L_n^+\bigr)$. Thus, the dual of  $\Cut_n$ is contained in $\underline{\BR}\bigl(\L_n^+(2)\bigr)$. 
We show the opposite inclusion.

Let  $E$ be the set of entry-wise exponentials of extremal rays of $\Cut_n$, viewed as $n \times n$ symmetric matrices with diagonal entries $1$.
Since every element of $\log E$ is the tree metric of a phylogenetic tree with only one positive-length edge, it follows that $E \subseteq \underline{\Delta}_n^+(\T_0) \subseteq \underline{\L}_n^+$. Since every matrix  in $E$ has rank at most 2, we have
$E \subseteq \underline{\L}_n^+(2)$.
Thus, any reduced bounded ratio $\alpha$ on $\L_n^+(2)$ satisfies $\alpha \cdot \beta \le 0$ for all $\beta \in \log E$, and hence for all $\beta \in \Cut_n$. 
\end{proof}

We close with a statement that strengthens Conjecture~\ref{conj:intro-at-most-2} for rank $2$ Lorentzian matrices.

\begin{conjecture}\label{conj:sub-free}
Let $\alpha \in \BR(\L_n^+)$ be any integral bounded ratio and $(p_{ij})_{1\le i\le j\le n}$ be the Hessian matrix of the quadratic form $\left(\sum_{1\leq i\leq n}a_ix_i\right)\left(\sum_{1\leq i\leq n}b_ix_i\right)$. Then 
\[
2^{\sum_{i=1}^n\alpha_{ii}}\Biggl(\prod_{\alpha_{ij}<0}p_{ij}^{-\alpha_{ij}}\Biggr)-\Biggl(\prod_{\alpha_{ij}>0}p_{ij}^{\alpha_{ij}}\Biggr)\]
	is a polynomial with nonnegative coefficients in $a_i$ and $b_i$. 
\end{conjecture}

For example, for $(a_1x_1+a_2x_2+a_3x_3)(b_1x_1+b_2x_2+b_3x_3)$ with $a_i,b_i \ge 0$, we have 
\[
\frac{p_{11}p_{23}}{p_{12}p_{13}}=\frac{(2a_1b_1)(a_2b_3+a_3b_2)}{(a_1b_2+a_2b_1)(a_1b_3+a_3b_1)} \le 2,
\]
and, in fact, the difference $2p_{12}p_{13}-p_{11}p_{23}=2a^2_1b_2b_3+2a_2a_3b_1^2$
has nonnegative coefficients.
Computations show that, for $n \le 6$, 
every primitive bounded ratio on $\L_n^+$ satisfies Conjecture~\ref{conj:sub-free}. 
A similar subtraction-free phenomenon for bounded ratios of Pl\"ucker coordinates over totally positive Grassmannians was observed in \cite[Conjecture 37]{BoFr}. For a connection to cluster algebras,  see \cite{Gekhtman-Greenberg-Soskin}.

\bibliography{biblio}
\bibliographystyle{amsalpha}

\end{document}